\documentclass[12pt,oneside,reqno,fleqn]{amsart}
\usepackage{amssymb}
\usepackage{bbm}
\usepackage{cases}
\usepackage{amsmath,mathtools}
\usepackage{graphicx}
\usepackage{mathrsfs}
\usepackage{stmaryrd}
\usepackage{color}
\usepackage{soul}
\usepackage[dvipsnames]{xcolor}
\usepackage{amsfonts}
\usepackage{amsmath,amssymb,amsthm}
\usepackage{libertine}
\usepackage[T1]{fontenc}
\usepackage[libertine,varbb]{newtxmath}
\usepackage[colorlinks,linkcolor=NavyBlue,citecolor=Gray,urlcolor=Periwinkle]{hyperref}
\usepackage[shortlabels] {enumitem}
\usepackage[nameinlink,capitalize]{cleveref}
\usepackage{orcidlink}
\usepackage{microtype}
\usepackage[textsize=tiny,disable]{todonotes}
\hypersetup{pdftitle={Stochastic processes with bounded and vanishing mean oscillation}, pdfauthor={Khoa L\^e}}
\numberwithin{equation}{section}
\newtheorem{theorem}{Theorem}[section]
\newtheorem{lemma}[theorem]{Lemma}
\newtheorem{proposition}[theorem]{Proposition}
\newtheorem{corollary}[theorem]{Corollary}
\theoremstyle{remark}

\newtheorem{example}[theorem]{Example}
\newtheorem{remark}[theorem]{Remark}
\newtheorem{definition}[theorem]{Definition}
\newtheorem{innercustomthm}{Condition}
\newenvironment{customcon}[1]
  {\renewcommand\theinnercustomthm{#1}\innercustomthm}
  {\endinnercustomthm}
\crefname{eqn}{Equation}{Equations}
\crefname{assumption}{Assumption}{Assumptions}
\crefname{innercustomthm}{Condition}{Conditions}
\crefrangelabelformat{innercustomthm}{#3#1#4-#5#2#6}
\newcommand\mI{{\mathbb I}}
\newcommand\cee{{\mathcal E}}
\newcommand\cff{{\mathcal F}}
\newcommand\cgg{{\mathcal G}}

\newcommand\cpp{{\mathcal P}}
\newcommand\mE{{\mathbb E}}
\newcommand{\X}{{\ensuremath{\mathbf{X}}}}
\newcommand{\XX}{\ensuremath{\mathbb{X}}}
\newcommand\E{\mE}
\newcommand\PP{{\mathbb P}}
\newcommand\mP{\PP}
\newcommand\les{\lesssim}

\newcommand{\R}{{\mathbb R}}
\newcommand{\mR}{\R}
\newcommand{\Rd}{{\R^d}}
\newcommand{\tand}{\quad\text{and}\quad}
\newcommand{\LL}{{\mathbb{L}}}

\newcommand{\1}{{\mathbf 1}}

\newcommand{\norm}[1]{{\left\vert\kern-0.25ex\left\vert\kern-0.25ex\left\vert #1 
    \right\vert\kern-0.25ex\right\vert\kern-0.25ex\right\vert}}
\DeclareMathOperator*{\esssup}{ess\,sup}
\renewcommand{\le}{\leq}
\renewcommand{\ge}{\geq}
\newcommand{\bmo}{{\mathrm{BMO}}}
\newcommand{\vmo}{{\mathrm{VMO}}}
\newcommand{\var}{{\mathrm{var}}}
\newcommand{\khoa}[2][]
{\todo[color=SeaGreen,caption={}, #1]{#2}}

\allowdisplaybreaks
\begin{document}
	\title{Quantitative John--Nirenberg inequality for  stochastic processes of  bounded mean oscillation}
	
\author{Khoa L{\^e} \orcidlink{0000-0002-7654-7139}}
\address{School of Mathematics, University of Leeds, U.K.
}
\email{k.le@leeds.ac.uk}
	\begin{abstract}
		Stroock and Varadhan in 1997 and  
		Geiss  in 2005 independently  introduced stochastic processes with bounded mean oscillation (BMO) and established their exponential integrability with some  unspecified exponential constant. This result is an analogue of the John--Nirenberg inequality for functions of bounded mean oscillation. In this work, we quantify the size of the exponential constant by the modulus of mean oscillation. Some new applications of BMO processes in rough stochastic differential equations, numerical approximations and regularization by noise are discussed.
		\bigskip
		

	\end{abstract}
	
	\maketitle
\section{Introduction} 
\label{sec:introduction}
	A real-valued locally integrable function $f$ defined on $\Rd$ is of bounded mean oscillation (BMO) if
	\begin{align*}
		\|f\|_{\bmo}:=\sup_Q\frac1{|Q|}\int_Q|f(x)-f_Q|dx<\infty,
	\end{align*}
	where the supremum is taken over all cubes $Q$ in $\Rd$, $|Q|$ denotes the Lebesgue measure of $Q$ and $f_Q=\frac1{|Q|}\int_Q f(x)dx$.
	For such function, 
	John and Nirenberg in \cite{MR131498} show that 
	\begin{align}\label{ineq.JNfunction}
		\sup_Q\frac1{|Q|}\int_Q e^{\lambda |f-f_Q|}dx<\infty
	\end{align}
	for some constant $\lambda>0$. The largest constant $\lambda$ such that \eqref{ineq.JNfunction} holds is denoted by $\lambda(f)$ and can be quantified by the distance between $f$ and the space of essentially bounded functions $L^\infty(\Rd)$ via the Garnett--Jones theorem \cite{MR506992} which asserts that  
	\begin{align}\label{eqn.disL}
		\frac1{C(d)}\frac1{\lambda(f)}\le\inf_{g\in L^\infty(\Rd)}\|f-g\|_\bmo\le C(d)\frac1{\lambda(f)}
	\end{align}
	for some constant $C(d)$.

	Stochastic processes of bounded mean oscillation are considered by Stroock and Varadhan \cite{MR2190038} and independently by S. Geiss in \cite{MR2136865}. 
	To give the precise definition, let $(\Omega,\cgg,\PP)$ be a  probability space equipped with a filtration $\{\cff_t\}_{t\ge0}$ satisfying the usual conditions. Let $\tau>0$ be a fixed number. 
	For each stopping time $S$, $\E_S$ denotes the conditional expectation with respect to $\cff_S$ and for each $G\in \cgg$, $\PP_S(G):=\E_S(\1_G)$.
	For each $m\in[1,\infty]$, $\|\cdot\|_m$ denotes the norm in $L^m(\Omega,\cgg,\PP)$.
	\begin{definition}\label{def.BMO}
		Let  $(V_t)_{t\in[0,\tau]}$ be a real valued adapted right continuous process with left limits (RCLL). $V$ is of \textit{bounded mean oscillation} (BMO) if 
		\begin{align}\label{eqn.defBMO}
			[V]_\bmo:=\sup_{0\le S\le T\le \tau}\|\E_S|V_T-V_{S-}|\|_\infty <\infty
		\end{align}		
		where the supremum is taken over all  stopping times $S,T$; $V_{s-}=\lim_{r\uparrow s}V_r$ and we set $V_{0-}=V_0$ by convention.
	\end{definition}
	\begin{remark}
		In \cite{MR2136865}, BMO processes are defined so that \eqref{eqn.defBMO} holds with $T=\tau$. This definition is equivalent to ours due to triangle inequality and the fact that $\E_S|V_\tau-V_{S}|=\lim_{\varepsilon\downarrow0}\E_S|V_\tau-V_{S+\varepsilon-}|$.
	\end{remark}
	\begin{remark}
		Herein, we focus on  estimations for moments of BMO processes and for this problem, there is no loss of generality when restricting to real valued processes. Indeed, if $Z$ is an adapted RCLL process taking values in some metric space $(\cee,d)$ such that 
		\[
			\sup_{S\le T\le \tau}\|\E_S d(V_{S-},V_T)\|_\infty<\infty
		\]
		then the processes $Z_{\cdot}=d(V_0,V_{\cdot})$ is a real valued BMO process as of \cref{def.BMO}. This is an immediate consequence of the triangle inequality. In fact, the maximal process $Z^*_t=\sup_{s\le t}|Z_s|$ is also of BMO, see \cref{prop.maximal} below.
	\end{remark}
	For BMO processes, Geiss shows in \cite[Theorem 1]{MR2136865} that
	\begin{align}\label{est.GeissJN}
		\sup_{s\in[0,\tau]}\|\E_s e^{\lambda\sup_{t\in[s,\tau]}|V_t-V_{s-}|}\|_\infty <\infty
	\end{align}
	for some constant $\lambda>0$, which is an analogue of the John--Nirenberg inequality \eqref{ineq.JNfunction} for BMO functions. 
	For continuous BMO processes, \eqref{est.GeissJN} was shown earlier in the last pages of the book \cite{MR2190038}. Stroock--Varadhan  called out John--Nirenberg inequality however they did not name BMO processes.
	The largest constant $\lambda$ such that \eqref{est.GeissJN} holds is denoted by $\lambda(V)$. Quantitative estimates for $\lambda(V)$ have not been considered, as far as the author's knowledge. Nevertheless, Varopoulos in \cite{MR599332} shows the following estimates for BMO martingales
	\begin{align}\label{eqn.probGJ}
		\frac{C_1}{\lambda(V)}\le\inf_{\psi\in L^\infty(\Omega)}\|V-\E_\cdot \psi\|_\bmo\le \frac{C_2}{\lambda(V)}.
	\end{align}
	for some constants $C_1,C_2>0$,	
	which is a probabilistic Garnett--Jones theorem.

	BMO processes and the John--Nirenberg inequality \eqref{est.GeissJN} have been utilized effectively in theory of singular integrals by Stroock \cite{MR341601}, in financial mathematics and backward SDEs by Geiss and coauthors in \cite{MR4092418,MR2136865,geiss2020riemann}. It is remarkable that BMO processes appear inconspicuously in many other problems,	in \cite{MR2377011,athreya2020well,LeSSL2,galeati2022solution} on regularization by noise phenomenon, in \cite{MR4125787,le2021taming,dareiotis2021quantifying} on strong convergence rate of numerical methods for stochastic differential equations, in \cite{FHL21} on rough stochastic differential equations.
	In these occurrences, BMO property was not identified  and the connection with John--Nirenberg inequality was not established.
	We give two examples of BMO processes which arise from these applications.
	\begin{example}\label{ex.bmo}
		(a) Let $(X^x_t)_{t\ge0, x\in\Rd}$ be a Markov process and let $f:[0,\tau]\times \Rd\to\R$ be a measurable function. Suppose that one has the uniform Krylov estimate
		\begin{align}\label{est.krylov}
			\sup_{0\le s\le t\le \tau}\sup_{x\in\Rd}\E|\int_s^t f(r,X^x_{r-s})dr|\le C
		\end{align}
		for some finite constant $C$.
		If furthermore the process $K_t:=\int_0^t f(r,X^0_r)dr$ is a.s. continuous then it is BMO. (Indeed, by Markov property, $\E_s|K_t-K_s|=\E|\int_s^t f(r,X^x_{r-s})dr|\big|_{x=X^0_s}\le C$ so that $K$ is BMO by  \cref{prop.s2S} herein.)
		Krylov estimate
		is an important tool in the study of stochastic differential equations (SDEs), see for instance \cite{rockner2021sdes,MR3563191}. 

		(b) Let $f:[0,1]\times\Rd \to\R$ be a Borel function which is uniformly bounded by $1$ and $B$ be a standard Brownian motion in $\Rd$. Define for each integer $n\ge1$, the process $V^n_t=\int_0^t [f(r,B_r)-f(r,B_{\lfloor nr\rfloor/n})]dr$, $t\in[0,1]$, which corresponds to the error of the quadrature rule approximating the functional $\int_0^t f(r,B_r)dr$. Then, $(V^n_t)_{t\in[0,1]}$ is BMO with $[V^n]_\bmo\le (N\log(n+1)/n)^{1/2}$ for some finite constant $N$. Indeed, it is shown by Dareiotis and Gerencsér in  \cite{MR4125787} (see Lemma 2.1 therein) that
		\begin{align}\label{ineq.DG}
			\esssup_\omega(\E_s|V^n_t-V^n_s|^2)^{1/2}\le  (N\log(n+1)/n)^{1/2} \text{ for every }s\le t\le 1.
		\end{align}
		Without the conditional expectation, \eqref{ineq.DG} is also obtained by Altmeyer in \cite{MR4303901}. The authors of both works rely on explicit moment computations and therefore are restricted to the second moment.
		Quadrature error estimates such as the above are directly related to the strong convergence rate of the Euler--Maruyama scheme for SDEs. 
		It is of important interests to have quadrature error estimates  for all moments.
		We will revisit these processes in \cref{ex.JN}. 
	\end{example}

	The current article provides two main contributions. 
	
	\textbf{I.} We provide a lower bound for $\lambda(V)$ in terms of the \textit{modulus of mean oscillation} of $V$, which is new even for BMO martingales (see \cref{rmk.sharpJN}). This kind of result is inspired by Portenko's formulation of the Khasminskii's lemma for increasing processes, see Chapter I.1 in \cite{MR1104660}. While  Garnett--John inequality \eqref{eqn.probGJ} is a beautiful theoretical result, our estimate is a practical one because moduli of mean oscillation are readily available in most applications and therefore can be applied directly. Based upon this estimate, we provide quantitative exponential integrability for processes with \textit{vanishing mean oscillation}, which is  inspired in part  by Lyons' precise estimate for multiplicative functionals in \cite{MR1654527}. 
	The obtained results are new and lie  outside the scope of \eqref{est.GeissJN} and \eqref{eqn.probGJ}.

	\textbf{II.} We bring to light the role of BMO processes and John--Nirenberg inequality in various problems in  regularization-by-noise phenomenon, numerical approximation for SDEs and rough stochastic differential equations. This is far from being mundane and we illustrate by three  applications, in which knowledge of BMO processes effectively shorten existing proofs and improve results. 
	The first one is about exponential moments of a class of  stochastic controlled rough paths considered in \cite{FHL21}. 
	In the second application, we explain that Davie's exponential estimate (\cite{MR2377011}) can be simplified and  deduced from a single estimate for the second moment.
	The last application is about strong convergence rate of the tamed Euler--Maruyama scheme for SDEs with integrable drifts. Such rate has been obtained previously in \cite{le2021taming} however only for small moments. Here, we are able to derive the same rate for all moments using  John--Nirenberg inequality in conjunction with the recent stability estimate of Galeati and Ling in \cite{galeati2022stability}. 

	\textit{Further literature.} Fefferman in \cite{MR280994} identified the space of BMO functions as the dual of the Hardy space $H_1$. 
	Getoor and Sharpe introduced among other things in \cite{MR305473} continuous time BMO martingales and established the duality of $H_1$ and BMO martingales, which is  a probabilistic analogue of Fefferman's earlier result. This duality for discrete time  martingales is due to Fefferman and Stein \cite{MR447953}, Garsia and  Herz \cite{MR331502,MR353447}.
	Results for continuous time BMO martingales are summarized in \cite{MR1299529}. See \cite{MR2390189,delbaen1997weighted,MR2136865} for  further applications in financial mathematics.

	\textit{Open problems.} Several challenging questions remain open. Geiss considered in \cite{MR2136865} weighted BMO and established a John--Nirenberg-type inequality for these processes. It is interesting to obtain analogous results presented herein to these processes. Another problem is to establish the Garnett--John theorem for BMO processes, extending \cite{MR599332}. Furthermore, Fefferman-type duality for BMO processes seems unexplored.

	We conclude the introduction with the layout of the paper. Results for BMO and VMO processes are presented in \cref{sec:BMO,sec:vmo_processes} respectively. We also provide another proof of Geiss' estimate \eqref{est.GeissJN}. The applications are discussed in \cref{sec:applications}. The appendix contains two auxiliary results which are well-known but are adjusted to our setting.  

\section{BMO processes} 
\label{sec:BMO}
	
	
	We recall \cref{def.BMO} of BMO processes.
	For each BMO process $V$, we define its \textit{modulus of mean  oscillation} $\rho(V):\{(s,t)\in[0,\tau]^2:s\le t\}\to[0,\infty)$  by
	\begin{align*}
		\rho_{s,t}(V)=\sup_{s\le S\le T\le t} \|\E_S|V_T-V_{S-}|\|_\infty, \quad 0\le s\le t\le \tau,
	\end{align*}
	where the supremum is taken over all stopping times $S,T$ satisfying $s\le S\le T\le t$.
	The function $\rho(V)$ is monotone in the sense that
	\begin{align}\label{def.monotone}
	 	\rho_{u,v}(V)\le \rho_{s,t}(V) \text{ whenever } [u,v]\subset[s,t].
	\end{align}
	We define the function $\kappa(V):\{(s,t)\in[0,\tau]^2:s\le t\}\to[0,\infty)$ by the relation
	\begin{align}\label{def.characteristic}
		\kappa_{s,t}(V)=\lim_{h\downarrow0}\sup_{s\le u\le v\le t,\,v-u\le h}\rho_{u,v}(V).
	\end{align}
	It is easy to see that the above limit always exists, that $\kappa(V)$ is monotone and $\kappa_{s,t}(V)\in[0,\rho_{s,t}(V)]$ for each $s\le t$. 

	An immediate consequence of \eqref{eqn.defBMO} is that all jumps of a BMO process are uniformly bounded by $\kappa$.
	\begin{proposition}\label{prop.jumps}
		If $V$ is BMO then $\|\sup_{t\in[0,\tau]}|V_t-V_{t-}|\|_\infty \le \kappa_{0,\tau}(V)$.
	\end{proposition}
	\begin{proof}
		By definitions, 
			$\|\E_S|V_T-V_{S-}|\|_\infty\le \rho_{0,\tau}(V)$
		for all stopping times in $[0,\tau]$.
		Taking $T=S$ yields that $\sup_{S}|V_S-V_{S-}|(\omega)\le \rho_{0,\tau}(V)$ for all $\omega$ in a set of full probability. For such $\omega$ and for each $r\in[0,\tau]$ such that $V_r(\omega)\neq V_{r-}(\omega)$, there is a stopping time $S$ such that $S(\omega)=r$ (Prop. 2.26 \cite{MR1121940}). With this stopping time, one deduces that $|V_r-V_{r-}|(\omega)\le \rho_{0,\tau}(V)$. The argument can be applied over arbitrary sub-intervals of $[0,\tau]$.
		Hence, for all $\omega$ in a set of full probability, for every rationals $s\le t$ in $[0,\tau]$ and every $r\in[s,t]$, we have $|V_r-V_{r-}|(\omega)\le \rho_{s,t}(V)$. This implies the result. 
	\end{proof}
	
	If $V$ is RCLL and has uniformly bounded jumps, one can replace the  stopping times in \eqref{eqn.defBMO} by deterministic times.
	\begin{proposition}\label{prop.s2S}
		Let $V$ be an adapted RCLL process and assume that
		\begin{align*}
			\sup_{s\le t\le  \tau}\|\E_s|V_t-V_{s-}|\|_\infty\le B
			\tand
			\sup_{t\in[0,\tau]}|V_t-V_{t-}|\le C
		\end{align*}
		for some finite constants $B,C$.
		Then for every stopping times $S\le T\le \tau$, one has
		\begin{align}\label{est.stoS}
			\E_S|V_{T}-V_{S-}|\le 2B+3C.
		\end{align}
		Consequently, $V$ is BMO.
	\end{proposition}
	\begin{proof}				
		Let $S\le \tau$ be a stopping time and ssume that $S$ takes finitely many values $\{s_1<\ldots<s_k\}$. We have 
		\begin{align*}
			\E_S|V_\tau-V_{S-}|
			&=\sum_{j}\E_S[|V_\tau-V_{S-}|\1_{(S=s_j)}]
			=\sum_{j}\1_{(S=s_j)}\E_{s_j}[|V_\tau-V_{s_j-}|]
			\le B.
		\end{align*}
		Every general stopping time  $S$ is the decreasing limit of a sequence of discrete stopping times $S^n$. Without loss of generality, we assume that $S^n\le \tau+1$ and define $V_t=V_\tau$ for all $t\ge \tau$.
		Then by triangle inequality 
		\begin{align*}
			\E_S[|V_{\tau+1}-V_{S-}|\wedge N]
			&\le \E_S[|V_{\tau+1}-V_{S^n-}|]+\E_S[|V_{S^n-}-V_{S-}|\wedge N]
			\\&\le B+\E_S[|V_{S^n-}-V_{S-}|\wedge N].
		\end{align*}		
		Note that $\lim_n V_{S^n-}=V_S$ so that by Fatou lemma and Lebesgue dominated convergence theorem,
		\begin{align*}
			\E_S[|V_{\tau+1}-V_{S-}|\wedge N]\le B+\E_S[|V_{S}-V_{S-}|\wedge N]\le B+C.
		\end{align*}
		Sending $N\to\infty$ yields 
		\begin{align}\label{tmp.sbc}
			\E_S|V_\tau-V_{S-}|\le B+C.
		\end{align}

		Let $S\le T\le \tau$ be  stopping times. 
		We have by triangle inequality that
		\begin{align*}
			\E_S|V_T-V_{S-}|
			\le \E_S|V_\tau-V_{S-}|+\E_S \E_T|V_\tau-V_{T-}|+\E_S|V_T-V_{T-}|.
		\end{align*}
		Hence, applying \eqref{tmp.sbc} and the assumptions, we obtain \eqref{est.stoS}.
	\end{proof}
	\begin{theorem}[John--Nirenberg inequality]\label{thm.JohnNirenberg}
		Let $V$ be a BMO process and let $r$ be a fixed number in $[0,\tau]$. Then 
		\begin{align}
			\|\E_r\sup_{r\le t\le \tau}|V_t-V_r|^p\|_\infty\le p!(11 \rho_{r,\tau}(V))^p \text{ for every integer }p\ge1,\label{JN.p}
			\\\shortintertext{and}
			\E_r e^{\lambda\sup_{r\le t\le \tau}|V_t-V_r|}<\infty
			\text{ for every }
			\lambda<(11\kappa_{r,\tau}(V))^{-1}.
			\label{JN.exp}
		\end{align}
	\end{theorem}
	
	\begin{remark}\label{rmk.sharpJN}
		Comparing with \eqref{est.GeissJN}, the estimate \eqref{JN.exp} is more precise because it relates  the range of the exponential constant $\lambda$ with the function $\kappa$, which leads to new exponential estimates in the following section.  
		The role of $\kappa$ in \eqref{JN.exp} is intrinsic to stochastic processes over finite time domains.
		To be more precise, we recall that a continuous martingale $(X_t)_{t\ge0}$ is BMO if 
		\begin{align*}
			[X]^2_{\bmo_2}:=\sup_S\|\E_S|X_\infty-X_S|^2\|_\infty<\infty,
		\end{align*}
		where the supremum is taken over all stopping times $S$. Os\c{e}kowski shows in \cite{MR3426632} that the inequality
		\begin{align}\label{est.sharpJN}
			\E e^{\lambda \sup_{t\ge0}|X_t-X_0|}\le \int_0^\infty e^{\lambda[X]_{\bmo_2}}e^{-t}dt \quad(\lambda>0)
		\end{align}
		is true and sharp, i.e. there is a martingale $X$ with $0<[X]_{\bmo_2}<\infty$ for which both sides are equal. One sees that $\kappa(X)$ plays no role whatsoever for BMO processes over infinite time horizon.
		In the other direction, let $(V_t)_{t\in[0,\tau]}$ be a continuous BMO martingale and define $X_t=V_{t\wedge \tau}$. Then $(X_t)_{t\ge0}$ is a continuous BMO martingale on the whole positive axis and is subjected to \eqref{est.sharpJN}. In particular, one sees that $\E e^{\lambda\sup_{t\in[0,\tau]}|V_t-V_0|}$ is finite if $\lambda[V_{\cdot\wedge \tau}]_{\bmo_2}<1$. Since $[V_{\cdot\wedge \tau}]_{\bmo_2}$ is comparable to $\rho_{[0,\tau]}(V)$ (by John--Nirenberg inequality), we deduce that $\E e^{\lambda\sup_{t\in[0,\tau]}|V_t-V_0|}$ is finite if $\lambda \rho_{[0,\tau]}(V)<c$
		for some universal constant $c$.
		This is much more restrictive than the condition $\lambda \kappa_{[0,\tau]}(V)<11$ provided by \cref{thm.JohnNirenberg}, especially for processes of vanishing mean oscillation (cf. \cref{thm.vmo}).
	\end{remark}

	Combining with \eqref{eqn.probGJ} and recalling the notation $\lambda(V)$ from the introduction, one immediately obtains
	\begin{corollary}
		There is a constant $C_3>0$ such that for all BMO martingales $V$, \begin{align*}
			\lambda(V)\ge (11 \kappa_{0,\tau}(V))^{-1}
			\tand
			\inf_{\psi\in L^\infty(\Omega)}\|V- \E_\cdot\psi\|_\bmo\le C_3 \kappa_{0,\tau}(V).
		\end{align*}
	\end{corollary}
	We revisit the example from the introduction and discuss the implication of \cref{thm.JohnNirenberg}.
	\begin{example}\label{ex.JN}
		We recall \cref{ex.bmo}.

		(a) From \cref{thm.JohnNirenberg}, we have
		\begin{align}\label{est.krylovp}
		 	\esssup_\omega\E_s\left(\sup_{t\in[s,\tau]}\Big|\int_s^t f(r,X^0_r)dr\Big|^p\right)\le p!(11C)^p
		 \end{align}
		for every integer $p\ge1$ and every $s\in[0,\tau]$.
		The John--Nirenberg inequality for BMO processes can thus be considered as a passage from uniform Krylov estimate to moment estimates of all orders. Such passage has been known previously only when $f$ is non-negative through Khasminskii's lemma, or when $f$ is a distribution through  the stochastic sewing lemma from \cite{le2018stochastic} 
		under some additional constraints on the modulus of mean  oscillation (see also Remark 5.3 in \cite{athreya2020well}). 
		
		(b) \cref{thm.JohnNirenberg} implies the following new estimate without any book-keeping  calculations of high moments
		\begin{align}\label{est.dgp}
			\esssup_\omega\left(\E_s\sup_{s\le t\le 1}\Big|\int_s^t [f(r,B_r)-f(r,B_{\lfloor nr\rfloor/n})]dr\Big|^p\right)\le p! (121N\log(n+1)/n)^{\frac p2}
		\end{align}
		for every integer $p\ge1$ and  every $s\le 1$.

		Later in \cref{thm.vmo}, we will see that when the modulus of mean  oscillation can be quantified, one can improve the growth constant $p!$ in \eqref{est.krylovp} and \eqref{est.dgp}. 
	\end{example}
	The proof of \cref{thm.JohnNirenberg} relies on the following two intermediate results.
	\begin{proposition}\label{lem.khasminskii}
		Let $(A_t)_{t\in[0,\tau]}$ be a BMO process which is non-decreasing.
		Then for every $r\in[0,\tau]$ and every $\lambda<(\kappa_{r,\tau}(A))^{-1}$, there is a finite constant $c=c(\lambda,\rho(A)|_{[r,\tau]^2})$ such that $\E_r e^{\lambda(A_\tau- A_r)}\le c$.
	\end{proposition}
	\begin{proof}
		It suffices to show the result for $r=0$.
		By assumption,  $\|\E_S(A_t-A_{S-})\|_\infty\le \rho_{s,t}(A)$ 
		for every times $s\le t\le \tau$ and every stopping time $s\le S\le t$. We apply the energy inequality, \cref{lem.energy}, to obtain that $\E_s[(A_t-A_s)^p]\le p!( \rho_{s,t}(A))^p$ for every $s\le t\le \tau$ and every integer $p\ge1$. 
		For each $\lambda<(\kappa_{0,\tau}(A) )^{-1}$, there is an $h_0>0$ such that $\lambda \rho_{s,t}(A)<1$ whenever $t-s\le h_0$. For such $s,t$, we have by Taylor's expansion that 
		\begin{align*}
			\|\E_s e^{\lambda(A_t-A_s)}\|_\infty\le (1- \lambda\rho_{s,t}(A))^{-1}<\infty.
		\end{align*}
		Now partition  $[0,\tau]$ by points $0=t_0<t_1<\ldots<t_n=\tau$ so that $\max_{1\le k\le n}(t_k-t_{k-1})\le h_0$. 
		Then
		\begin{align*}
			\E_r e^{\lambda(A_\tau-A_r)}
			&=\E_r e^{\lambda(A_{t_{n-1}}-A_r)}e^{\lambda(A_{t_n}-A_{t_{n-1}})}
			=\E_re^{\lambda(A_{t_{n-1}}-A_r)}\E_{t_{n-1}} e^{\lambda(A_{t_n}-A_{t_{n-1}})}
			\\&\le  \E_re^{\lambda(A_{t_{n-1}}-A_r)}\|\E_{t_{n-1}} e^{\lambda(A_{t_n}-A_{t_{n-1}})}\|_\infty.
		\end{align*}
		Iterating the previous inequality yields
		\begin{align*}
			\|\E_r e^{\lambda(A_\tau-A_r)}\|_\infty
			&\le \|\E_{r} e^{\lambda(A_{t_j}-A_r)}\|_\infty  \prod_{k=j+1}^n \|\E_{t_{k-1}} e^{\lambda(A_{t_k}-A_{t_{k-1}})}\|_\infty,
		\end{align*}
		where $j$ is such that $t_{j-1}\le r<t_{j}$. This implies the bound
		\begin{align*}
			\|\E_r e^{\lambda(A_\tau-A_r)}\|_\infty
			&\le 
			\prod_{k=1}^n (1- \lambda \rho_{t_{k-1},t_k}(A))^{-1},
		\end{align*}
		which yields the result.
	\end{proof}
	\begin{proposition}\label{prop.maximal}
		Let $V$ be a BMO process. Define $V^*_t=\sup_{s\le t}|V_s-V_0|$. Then $V^*$ is BMO with $\rho(V^{*})\le 11 \rho(V)$. 
	\end{proposition}
	\begin{proof}
		Fix $s\le t$. 
		For stopping times $s \le  S\le T\le t$, we have
		\begin{align*}
			\E_S|V_T-V_{S-}|\le c \text{ with } c=\rho_{s,t}(V).
		\end{align*}
		We define $D_{s,t}=\sup_{s\le r\le t}|V_r-V_{s-}|$ and apply \cref{lem.upcrossing} to obtain that
		\begin{align*}
			\beta\PP_s\left((D_{s,t}- \beta)^+\ge \alpha\right)\le \beta\PP_s(D_{s,t}\ge \alpha+\beta)\le \rho_{s,t}(V)\PP_s( D_{s,t}\ge \alpha)
		\end{align*}
		for every $\alpha,\beta>0$. 
		It follows that $\beta\E_s[(D_{s,t}- \beta)^+]\le c\E_s D_{s,t}$. Choosing $\beta=2 c$, we have $2 c\E_s(D_{s,t}-2 c)\le c\E_s D_{s,t}$, that is $\E_s \sup_{s\le r\le t}|V_r-V_{s-}|\le 4 \rho_{s,t}(V)$.
		Combining with the elementary estimate $V^*_t-V^*_{s-}\le \sup_{s\le r\le t}|V_t-V_{s-}|$, we obtain that
		\[
			\E_s |V^*_t-V^*_{s-}|\le 4 \rho_{s,t}(V).
		\]
		We also have $V^*_t-V^*_{t-}\le  V_t-V_{t-}$ for every $t\le \tau$ and hence an application of \cref{prop.s2S} yields the result.
	\end{proof}
	\begin{proof}[\textbf{Proof of \cref{thm.JohnNirenberg}}]
		It suffices to show the result for $r=0$. Define $A_t=\sup_{s\in[0,t]}|V_s-V_0|$. By \cref{prop.maximal}, $A$ is BMO and $\rho(A)\le 11 \rho(V)$. We obtain \eqref{JN.p} and \eqref{JN.exp} by applying \cref{lem.energy,lem.khasminskii} respectively.
	\end{proof}
\section{VMO processes} 
\label{sec:vmo_processes}
	\begin{definition}
		A BMO process $(V_t)_{t\in[0,\tau]}$ is VMO if $\kappa_{0,\tau}(V)=0$.
	\end{definition}
	An immediate consequence of \cref{prop.jumps} is that every VMO process has continuous sample paths.
	It is straightforward to see that the class of VMO processes starting from $0$ forms a closed subspace of the space of BMO processes starting from 0. 
	To quantify the regularity of VMO processes, we propose two additional subclasses.

	\begin{definition}
		Let $(V_t)_{t\in[0,\tau]}$ be a VMO process, $p\in[1,\infty)$ and  $\alpha\in(0,1]$ 
		be some fixed numbers. We say that $V$ is $\vmo^{p-\var}$ if $\rho(V)$ has finite $p$-variation over $[0,\tau]$, that is
		\begin{align*}
			[V]_{\vmo^{p-\var};[0,\tau]}:=\left( \sup_{\pi\in\cpp([0,\tau])}\sum_{[s,t]\in \pi}|\rho_{s,t}(V)|^{p}\right)^{1/p}<\infty
		\end{align*}
		where $\cpp([0,\tau])$ is the set of all partitions on $[0,\tau]$.
		We say that $V$ is $\vmo^\alpha$ if 
		\begin{align*}
			[V]_{\vmo^\alpha;[0,\tau]}:=\sup_{0\le s<t\le \tau}\frac{\rho_{s,t}(V)}{(t-s)^\alpha}<\infty.
		\end{align*}
	\end{definition}
	It is evident that $\vmo^\alpha\subset \vmo^{1/\alpha-\var}$ and that
	\begin{align*}
		[V]_{\vmo^{1/\alpha-\var};[0,\tau]}\le \tau^{\alpha} [V]_{\vmo^{\alpha};[0,\tau]}.
	\end{align*}
	
	When $\alpha>1$ and $0<p<1$, the spaces  $\vmo^\alpha$ and $\vmo^{p-\var}$ contain only constant processes. This can be verified rather directly or alternatively using \eqref{est.Vcontrol} below.
	\begin{proposition}	\label{prop.control}
		Let $(V_t)_{t\in[0,\tau]}$ be a process in $\vmo^{p-\var}$ and define the function
		\begin{align*}
			(s,t)\mapsto w_{s,t}(V):=([V]_{\vmo^{p-\var};[s,t]})^p.
		\end{align*}
		Then $w(V):\{(s,t)\in[0,\tau]^2:s\le t\}\to[0,\infty)$ is a control, i.e. $w(V)$ is continuous and satisfies  $w_{s,u}(V)+w_{u,t}(V)\le w_{s,t}(V)$ whenever $s\le u\le t$ (super-additivity). 
	\end{proposition}
	\begin{proof}
		That $w_{s,u}(V)+w_{u,t}(V)\le w_{s,t}(V)$ whenever $s\le u\le t$ is evident from definitions. Next, we  verify that $w(V)$ is continuous in several steps.

		\textit{Step 1.} We explain that $\rho(V)$ is continuous. Whenever $s\le u\le t$, we have by triangle inequality that
		\begin{align*}
			\|\E_s|V_t-V_s|\|_\infty \le\|\E_s|V_u-V_s|\|_\infty+\|\E_u|V_t-V_u|\|_\infty.
		\end{align*}
		This implies that 
		\begin{align}\label{tmp.rhotriangle}
		 	\rho_{s,t}(V)\le \rho_{s,u}(V)+\rho_{u,t}(V).
		\end{align}
		Hence,
		\begin{gather*}
			|\rho_{s,t}(V)-\rho_{s+,t}(V)|=\lim_{u\downarrow s}(\rho_{s,t}(V)-\rho_{u,t}(V))
			\le\lim_{u\downarrow s}\rho_{s,u}(V)=0
			\\\shortintertext{and}
			|\rho_{s-,t}(V)-\rho_{s,t}(V)|=\lim_{r\uparrow s}(\rho_{r,t}(V)-\rho_{s,t}(V))
			\le\lim_{r\uparrow s}\rho_{r,s}(V)=0
		\end{gather*}
		which show that $\rho(V)$ is continuous in the former argument.
		Similarly, one can show continuity in the later argument. Thus $\rho(V)$ is continuous. 
		
		\textit{Step 2.} We show that $w(V)$ is continuous from the inside, i.e. $w_{s,t}(V)=w_{s+,t-}(V)$. 
		Fix $s< u< t$ and a small number $h>0$. We have by super-additivity
		\begin{align*}
			w_{s+h,u-h}(V)+w_{u+h,t-h}(V)\le w_{s+h,t-h}(V).
		\end{align*}
		Sending $h$ to $0$, we see that $\bar w_{s,t}:=w_{s+,t-}(V)$ is super-additive. From the estimate $\rho_{s+h,t-h}(V)\le w_{s+h,t-h}(V)$, we also have $\rho_{s,t}(V)\le \bar w_{s,t}$. From definition of $w(V)$, this implies that $w(V)\le \bar w$. It is obvious that $\bar w\le w(V)$ and hence $\bar w=w(V)$, showing continuity from the inside.

		\textit{Step 3.} We show that $w_{s,s+}(V)=0$. From $w_{s,u}(V)+w_{u,t}(V)\le w_{s,t}(V)$, we send $u\downarrow s$ to get $w_{s,s+}(V)+w_{s+,t}(V)\le w_{s,t}(V)$. Using continuity from the inside, we obtain the claim

		\textit{Step 4.} We show that $w(V)$ is continuous from the outside, i.e. $w_{s,t}(V)=w_{s-,t+}(V)$. We fix $s<t$ and $h,\varepsilon>0$ and consider $\pi=\{t_i\}_{i=1}^n\in \cpp([s,t+h])$ such that
		\begin{align*}
			\sum_{i=1}^n |\rho_{t_i,t_{i+1}}(V)|^p> w_{s,t+h}(V)-\varepsilon.
		\end{align*}
		Let $j$ be such that $t_j< t\le t_{j+1}$. From the above inequality, we have
		\begin{align*}
			w_{s,t}(V)+|\rho_{t_j,t_{j+1}}(V)|^p-|\rho_{t_j,t}(V)|^p+w_{t,t+h}(V)>w_{s,t+h}(V)-\varepsilon.
		\end{align*}
		We send $h$ to $0$, noting that (by \eqref{tmp.rhotriangle}) $|\rho_{t_j,t_{j+1}}(V)|^p-|\rho_{t_j,t}(V)|^p\lesssim \rho_{t,t_{j+1}}(V)\lesssim \rho_{t,t+h}(V)$  vanishing, to obtain that $w_{s,t}(V)+w_{t,t+}(V)\ge w_{s,t+}(V)$. By the previous step, we have $w_{s,t}(V)\ge w_{s,t+}(V)$. The reverse inequality is obvious by monotonicity so that $w_{s,t}(V)= w_{s,t+}(V)$. In an analogous way, one has $w_{s,t}(V)= w_{s-,t}(V)$, and hence continuity from the outside.		
	\end{proof}

	As an immediate consequence, we have for each $V$ in $\vmo^{p-\var}$ that
	\begin{align}\label{est.Vcontrol}
		\|\E_s|V_t-V_s|\|_\infty\le |w_{s,t}(V)|^{1/p} \quad \forall\, 0\le s\le t\le \tau.
	\end{align}
	\begin{theorem}\label{thm.vmo}
		Let $(V_t)_{t\in[0,\tau]}$ be VMO. Then 
		\begin{align}\label{exp.vmo}
			\E_r e^{\lambda\sup_{t\in[r,\tau]}|V_t-V_r|}<\infty \text{ for every } r\in[0,\tau] \text{ and every } \lambda>0.
		\end{align}
		If $V$ belongs to $\vmo^{p-\var}$ then
		\begin{align}\label{exp.vmoa}
			\sup_{r\in[0,\tau]}\E_r e^{\lambda\sup_{t\in[r,\tau]}|V_t-V_r|}\le 2^{1+ (22\lambda )^{p}w_{0,\tau}(V)} \text{ for every } \lambda>0.
		\end{align}
	\end{theorem}
	\begin{proof}
		The estimate \eqref{exp.vmo} is a direct consequence of \cref{thm.JohnNirenberg}.

		Suppose now that $V$ is  $\vmo^{p-\var}$ and put $\rho= 11 w(V) ^{1/p}$. 
		We define $t_0=0$ and for each integer $k\ge1$,
		\begin{align*}
			t_k=\sup\{t\in [t_{k-1},\tau]: \lambda  \rho_{t_{k-1},t}\le 1/2\}.
		\end{align*}
		By continuity of $w$, we have $\lambda  \rho_{t_{k-1},t_k}=1/2$ for $k=1,\ldots,n-1$ and $\lambda  \rho_{t_{n-1},t_n}\le 1/2$. By definition of controls, we have
		\begin{align*}
			\frac{n-1}{(22 \lambda )^{p}}\le \sum_{k=1}^n w_{t_{k-1},t_k}(V)\le w_{0,\tau}(V),
		\end{align*}
		which yields $n\le 1+ (22\lambda )^{p}w_{0,\tau}(V)$.
		Fix $r\in[0,\tau]$ and define $A_t=\sup_{s\in[r,t]}|V_s-V_r|$. \cref{prop.maximal} shows that $A$ is BMO with $\rho(A)\le \rho$. Let $j$ be such that $t_{j-1}\le r<t_j$.
		Following the proof of \cref{lem.khasminskii}, we have 
		\begin{align*}
			\E_r e^{\lambda A_\tau}\le \prod_{k=j}^n (1- \lambda \rho_{t_{k-1}\vee r,t_k})^{-1}
			\le 2^n.
		\end{align*}
		These estimates imply \eqref{exp.vmoa}.		
	\end{proof}
	\begin{corollary}\label{cor.moment}
		Let $(V_t)_{t\in[0,\tau]}$ be $\vmo^{p-\var}$. 
		If $p\in(1,\infty)$, then there are  constants $c_p, C_p$ such that for all $\lambda$ satisfying $\lambda (w_{0,\tau}(V))^{\frac1{p-1}}  <c_p$,		
		\begin{align}\label{est.supervmo}
		 	\E\exp \left(\lambda\sup_{t\le \tau}|V_t-V_0|^{\frac p{p- 1}}\right)<\infty,
		\end{align}
		and for all real number $m\ge1$,
		\begin{align}\label{est.momentV}
			\|\E_s [\sup_{r\in[s,t]}|V_r-V_s|^m]\|_\infty \le C_p  \Gamma\big(m(1- 1/p)+1\big)\big(w_{s,t}(V)\big)^{m/p} ,
		\end{align}
		where $\Gamma(z)=\int_0^\infty u^{z-1}e^{-u}du$ is the Gamma function.

		If $p=1$ then 
		\begin{align}\label{est.v1}
			\PP\left(|V_t-V_s|\le 22w_{s,t}(V) \text{ for all } s\le t\le \tau\right)=1.
		\end{align}
	\end{corollary}
	\begin{proof}	
		Consider first the case $p\in(1,\infty)$.
		Define $Z=\sup_{t\le \tau}|V_t-V_0|$ and $a=1/\alpha$. By Chebyshev inequality and \eqref{exp.vmoa}, we have
		\begin{align*}
			\PP(Z>x)=\PP(e^{\lambda Z}>e^{\lambda x})\le e^{-\lambda x}\E e^{\lambda Z}\le C e^{-\lambda x+c \beta \lambda^{p}}
		\end{align*}
		where $\beta=w_{0,\tau}(V)$
		and $c=c(p), C=C(p)$ are some universal positive constants. One can optimize in $\lambda$ to obtain that for every $x$ bounded away from $0$,
		\begin{align*}
			\PP(Z>x)\le C e^{-c_p \beta^{-\frac{1}{p-1}} x^{p'}}, \quad\frac1p+\frac1{p'}=1,
		\end{align*}
		where $c_p,C$ are some other positive constants. In view of the layer cake representation
		\begin{align*}
		 	\E e^{\lambda Z^{p'}}=\lambda p'\int_0^\infty e^{\lambda x^{p'}}x^{p'-1}\PP(Z>x)dx,
		\end{align*}
		we see that $\E e^{\lambda Z^{p'}}$ is finite if $\lambda \beta^{1/(p-1)}<c_p$. We obtain \eqref{est.supervmo} by observing that $p'=p/(p-1)$.

		The estimate \eqref{est.momentV} is obtained in an analogous way. Define $Y=\sup_{r\in[s,t]}|V_r-V_s|$. Reasoning as previously,
		\begin{align*}
			\PP_s(Y>x)\le C e^{-c_p \beta^{-\frac1{p-1}}x^{p'}}.
		\end{align*}
		By the layer cake representation,
		\begin{align*}
			\E_s Y^m=m\int_0^\infty x^{m-1}\PP_s(Y>x)dx
			\le C m\int_0^\infty x^{m-1}e^{- c_p \beta^{-1/(p-1)}x^{p'}}dx.
		\end{align*}
		After the change of variable $y=c_p \beta^{-1/(p-1)}x^{p'}$, using the identity $\Gamma(z+1)=z \Gamma(z)$, we arrive at \eqref{est.momentV}.

		In the case $p=1$, a similar argument with Chebyshev inequality and \eqref{exp.vmoa} leads to
		\begin{align*}
			\PP(|V_t-V_s|>x)\les e^{-\lambda x+22w_{s,t}(V)\lambda} 
		\end{align*}
		for every $x>0$ and $\lambda>0$.
		When $x>22w_{s,t}(V)$, we can send $\lambda\to\infty$ to obtain that $\PP(|V_t-V_s|>x)=0$. This implies that $\PP(|V_t-V_s|\le 22w_{s,t}(V))=1$. Since $V$ has continuous sample paths, this implies \eqref{est.v1}.
	\end{proof}
	Estimate \eqref{est.momentV} is inspired by the precise estimate from Lyons' first extension theorem (inequality (2.21) in \cite{MR1654527}), which is obtained through the so-called neo-classical inequality. Note that there is a smallness condition in the aforementioned article, which is not present in \cref{cor.moment}. 

	We note that \eqref{est.supervmo} cannot be derived from the Garnett--John inequality \eqref{eqn.probGJ}. Let $\lambda_{p'}(V)$ denote the largest constant such that \eqref{est.supervmo} holds. An interesting open question  is whether one could characterize the size of $\lambda_{p'}(V)$ by the distance of $V$ to a certain subspace of $\vmo^{p-\var}$.  
\section{Applications} 
\label{sec:applications}
\subsection{Rough stochastic differential equations} 
\label{sub:rough_stochastic_differential_equations}
	In \cite{FHL21}, the authors consider a hybrid rough stochastic differential equation of the type
	\begin{align}\label{eq.rsde}
		d Y_t=b_t(Y_t)dt+\sigma_t(Y_t)dB_t+(f_t,f'_t)(Y_t)d\X_t
	\end{align}
	where $B$ is a standard Brownian motion and $\X=(X,\XX)$ is a H\"older rough path. The coefficients $b,\sigma,f,f'$ are progressively measurable and regular in the $Y$-component.
	Under some natural regularity conditions, \cite{FHL21} shows that \eqref{eq.rsde} has a unique continuous solution in a certain class of stochastic controlled rough paths denoted by $\mathbf{D}_X^\alpha L_{m,\infty}$ for some $\alpha\in(1/3,1/2]$ and $m\ge2$. 
	Such processes are adapted and satisfy
	\begin{align}\label{est.solRSDE}
		\sup_{s<t\le \tau}\frac{(\|\E_s|Y_t-Y_s|^m\|_\infty)^{1/m}}{(t-s)^\alpha}<\infty,
	\end{align}
	together with some controllness conditions.
	As is shown in \cite{FHL21}, this class of stochastic controlled rough paths are stable under composition with smooth vector fields and rough stochastic integration. 
	Because most of these properties are irrelevant for our purpose, we refer to the cited reference for further details.

	Based upon earlier sections, any continuous adapted process satisfying the property \eqref{est.solRSDE} is $\vmo^\alpha$. Hence, such stochastic controlled rough paths are subjected to the John--Nirenberg inequality.
	Although \cite{FHL21} also discusses exponential estimates for the solution of \eqref{eq.rsde} by means of Lyons' multiplicative functionals, their result comes with some additional restrictions on $m,\alpha$ and the connection with VMO processes was not present there.
	On the other hand, our results actually accommodate  minimal conditions that $\alpha\in(0,1]$ and $m=1$.
	\begin{theorem}\label{thm.expSCRP}
		Let $(Y_t)_{t\in[0,\tau]}$ be a continuous process in $\mathbf{D}^\alpha_XL_{1,\infty}$ (see Section 3 of \cite{FHL21} for the precise definition) for some $\alpha\in(0,1]$. Then $Y$ is $\vmo^\alpha$ and
		\begin{align*}
			\E e^{\lambda\sup_{t\in[0,\tau]}|Y_t-Y_0|}\le  2^{1+(22[Y]_{\vmo^\alpha}\lambda)^{1/\alpha}\tau}
			\text{ for every }\lambda>0.
		\end{align*}
	\end{theorem}
	\begin{proof}
		As discussed earlier, $Y$ satisfies \eqref{est.solRSDE} with $m=1$. By \cref{prop.s2S} and sample path continuity, $Y$ is necessarily $\vmo^\alpha$. The exponential estimate is a direct consequence of \eqref{exp.vmoa}.
	\end{proof}

	Another class of processes introduced in \cite{FHL21} is $C^\alpha L_{m,\infty}$ (with $\alpha\in(0,1]$ and $m\ge1$) consisting of adapted processes $Y$ such that
	\begin{align*}
		\sup_{t\in[0,\tau]}\|Y_t\|_m+\sup_{s<t\le \tau}\frac{(\|\E_s|Y_t-Y_s|^m\|_\infty)^{1/m}}{(t-s)^\alpha}<\infty.
	\end{align*}
	We can classify these classes as VMO processes in the following way.
	\begin{proposition}
		Let $\alpha\in(0,1]$ and $m\in[1,\infty]$ be some fixed numbers and $(Y_t)_{t\in[0,\tau]}$ be a continuous adapted process. $Y$ belongs to
		$C^\alpha L_{m,\infty}$ if and only if $Y_0$ is $L_m$-integrable and $Y$ is $\vmo^\alpha$.
	\end{proposition}
	\begin{proof}
		Straightforward from \cref{prop.s2S}.
	\end{proof}
\subsection{Davie's estimates} 
\label{sub:davie_s_estimates}
	Let $(B_t)_{t\ge0}$ be a standard Brownian motion and $g:[0,\tau]\times\Rd\to\R$ be a bounded Borel measurable function such that $|g(r,y)|\le 1$ for all $(r,y)\in[0,\tau]\times\Rd$. Davie shows in \cite{MR2377011} that for any even integer positive integer $m$ and $x\in\Rd$,
	\begin{align}\label{est.davie}
		\E\left(\int_0^1\left[g(t,B_t+x)-g(t,B_t)\right]dt\right)^m\le C^m\Gamma(m/2+1)|x|^m,
	\end{align}
	where $C$ is an absolute constant. 
	This inequality exhibits the regularization effect of Brownian motion through temporal integration. Such regularization effect is developed further into the framework of nonlinear Young integration by Catellier and Gubinelli in \cite{MR3505229}. 
	Davie's estimate is an important one and has  been reproduced in different forms under other conditions and setups \cite{MR3570126,rockner2021sdes,rezakhanlou2014regular,LeSSL2,le2018stochastic,athreya2020well}.

	Davie shows \eqref{est.davie} by first expanding the moment into an  iterated multiple integral. The lack of regularity in $g$ is compensated by the smooth density of the Brownian motion through integration by parts. This procedure produces a sum of $2^{p-1}$ iterated multiple integrals involving derivatives of the Gaussian density. He then estimates each of these multiple integrals carefully to obtain \eqref{est.davie}. Davie's proof is beautiful yet  intricate because of its analysis of high moments.
	We now explain how  John--Nirenberg inequality in \cref{thm.JohnNirenberg,thm.vmo} could be utilized in this context.
	Indeed, following Davie's proof in \cite{MR2377011}, one has
	\begin{align*}
	 	\E\left(\int_s^t\left[g(r,B_r+x)-g(r,B_r)\right]dr\right)^2\le C^2|x|^2(t-s).
	\end{align*}
	Since Brownian motion has independent increments, we can upgrade the above inequality to the following estimate
	\begin{align*}
	 	\E_s\left(\int_s^t\left[g(r,B_r+x)-g(r,B_r)\right]dr\right)^2\le C^2|x|^2(t-s).
	\end{align*}
	This shows that the process $V_t=\int_0^t\left[g(r,B_r+x)-g(r,B_r)\right]dr$ is $\vmo^{1/2}$ with  $[V]_{\vmo^{1/2}}\le C |x|$. 
	The estimate \eqref{est.momentV} gives
	\begin{align}\label{est.supDavie}
		\E\left(\sup_{t\in[0,1]}\Big|\int_0^t\left[g(r,B_r+x)-g(r,B_r)\right]dr\Big|\right)^m\le C^m \Gamma(m/2+1)|x|^m.
	\end{align}
	This estimate is comparable to or perhaps stronger than \eqref{est.davie} because of the supremum in its left-hand side.

\subsection{Quadrature error estimates and strong convergence rate of Euler method} 
\label{sub:strong_convergence_rate_of_tamed_euler_maruyama_scheme}
	Consider the stochastic different equation 
	\begin{align}\label{eq.sde}
		dX_t=b(t,X_t)dt+\sigma(t,X_t)dB_t, \quad X_0=x_0, \quad t\in[0,1],
	\end{align}
	where $d\ge1$, $b: [0,1]\times\mR^d\rightarrow\mR^d$, $\sigma:[0,1]\times\Rd\to \Rd\times\Rd$ are  Borel measurable functions, $(B_t)_{t\geq0}$ is a $d$-dimensional standard Brownian motion defined on some complete filtered probability space $(\Omega,\cff,(\cff_t)_{t\geq 0},\mP)$ and $x_0$ is an $\cff_0$-random variable.
	The tamed Euler--Maruyama scheme associated to \eqref{eq.sde} is 
	 \begin{equation}\label{eqn.EMscheme}
		 dX_t^n=b^n(t,X^n_{k_n(t)})dt+ \sigma(t,X^n_{k_n(t)}) dB_t,\quad X_0^n=x_0^n, \quad t\in[0,1],
	 \end{equation}
	 where $x^n_0$ is a $\cff_0$-random variable and $b^n$ is an approximation of the vector field $b$ and
	 \begin{align*}
	 	k_n(t)=\frac jn \text{ whenever } \frac jn\le t<\frac {j+1}n \text{ for some integer } j\ge0.
	 \end{align*}
	 We note that \eqref{eqn.EMscheme} with the choice $b^n=b$ is the usual Euler--Maruyama scheme, which, however, is not well-defined for a merely integrable function $b$ even when $b$ is replaced by $b\1_{(|b|<\infty)}$. This is because  the simulation for the usual Euler--Maruyama scheme may enter a neighborhood of a singularity of $b$, making the scheme unstable and uncontrollable.

	The recent article \cite{le2021taming} establishes strong rate of convergence of the tamed Euler--Maruyama scheme \eqref{eqn.EMscheme} to \eqref{eq.sde} under some integrability condition of the drift $b$.
	To state their result, we first recall some notation from \cite{le2021taming}.
	Let $p,q\in[1,\infty]$ be some fixed parameters.
	$L_p(\Rd)$ and $L_p(\Omega)$ denote the Lebesgue spaces respectively on $\Rd$ and $\Omega$.
 	For each $\nu\in\R$, $L_{\nu,p}(\Rd):=(1-\Delta)^{-\nu/2}\big(L^{ p}(\mR^d)\big)$ is the usual Bessel  potential space on $\Rd$ equipped with the norm
	$\|f\|_{L_{\nu, p}(\mathbb{R}^d)}:=\|(\mI-\Delta)^{\nu/2}f\|_{L_{ p}(\mathbb{R}^d)}$,
	where  $(\mI-\Delta)^{\nu/2}f$ is defined through Fourier's  transform. $\LL^q_{\nu,p}([0,1])$ denotes the space of measurable function $f:[0,1]\to L_{\nu,p}(\Rd)$ such that $\|f\|_{\LL^q_{\nu,p}([0,1])}$ is finite. Here, for each $s,t\in[0,1]$ satisfying $s\le t$, we denote
	 \begin{align*}
	  	\|f\|_{\LL^q_{\nu,p}([s,t])}:=\left(\int_s^t\|f(r,\cdot)\|_{L_{\nu,p}(\Rd)}^qdr\right)^{\frac1q}
	 \end{align*}
 with obvious modification when $q=\infty$.  When $\nu=0$, we simply write $\LL^q_p([0,1])$ instead of $\LL^q_{0,p}([0,1])$.
 In particular, $\LL^q_p([0,1])$ contains Borel measurable functions $f:[0,1]\times\Rd\to\R$ such that $\int_0^1\left[\int_\Rd|f(t,x)|^pdx\right]^{q/p}dt$ is finite.
 	As in \cite{le2021taming}, we assume the following conditions. 
	  \begin{customcon}{$\mathfrak{A}$}\label{con.A}
  	The diffusion coefficient $\sigma$ is a $d\times d$-matrix-valued measurable function on $[0,1]\times\Rd$ . There exists a constant $K_1\in[1,\infty)$ such that
  	for every $s\in[0,1]$ and $x\in\Rd$
  	\begin{align}\label{ellptic-con}
  		K_1^{-1}I\le (\sigma \sigma^*)(s,x)\le K_1I,
  	\end{align}
  	where $I$ denotes the identity matrix.
  	Furthermore, the following conditions hold.
  	\begin{enumerate}[label=$\mathbf{\arabic{*}}$., ref=$\mathbf{\arabic*}$]
  	 	\item\label{con.Aholder} There are constants $\alpha\in(0,1]$ and $K_2\in(0,\infty)$  such that for every $s\in[0,1]$ and $x,y\in\Rd$
  	 	\begin{align*}
  	 		|(\sigma \sigma^*)(s,x)-(\sigma \sigma^*)(s,y)|\le K_2|x-y|^\alpha.
  	 	\end{align*}
  	 	
  	 	\item\label{con.ASobolev} $\sigma(s,\cdot)$ is weakly differentiable for a.e. $s\in[0,1]$ and there are constants $p_0\in[2,\infty)$, $q_0\in(2,\infty]$ and $K_3\in(0,\infty)$  such that
	  	\begin{align*}
	  		\frac d{p_0}+\frac2{q_0}<1
	  		\tand\|\nabla \sigma\|_{\LL^{q_0}_{p_0}([0,1])}\le K_3.
	  	\end{align*}
  	 \end{enumerate}
  \end{customcon}
  \begin{customcon}{$\mathfrak{B}$}\label{con.B}	
  	$x_0$ belongs to $L_p(\Omega,\cff_0)$ and
  	$b$ belongs to $\LL^q_p([0,1])$ for some $p,q\in[2,\infty)$ satisfying $\frac dp+\frac2q<1$.
	For each $n$, $x^n_0$ belongs to $L_p(\Omega,\cff_0)$ and $b^n$ belongs to $\LL^q_p([0,1])\cap \LL^q_\infty([0,1])$ with $p,q$ as above.
	Furthermore, there exist finite positive constants $K_4,\theta$ and continuous controls $\{\mu^n\}_n$ such that $\sup_{n\ge1}(\|b^n\|_{\LL^q_p([0,1])}+\mu^n(0,1))\le K_4$ and
  		\begin{align}\label{con.locGK}
  		  	(1/n)^{\frac12-\frac1q}\|b^n\|_{\LL^q_\infty([s,t])}\le \mu^n(s,t)^\theta \quad \forall\ t-s\le1/n.	
  		\end{align}
  \end{customcon}
  \begin{definition}\label{def.rate}
  	Let $\lambda>0$ be a fixed number which is sufficiently large.
  	Let $U=(U^1,\ldots,U^d)$ where for each $h=1,\ldots,d$, $U^h$ is the solution to the following equation
  	\begin{align}\label{eqn.uKol}
  	\partial_t U^h+\sum_{i,j=1}^d\frac{1}{2}(\sigma \sigma^*)^{ij}\partial^2_{ij} U^h+b^{h}\cdot\nabla U^h= \lambda U^h-b^{h},\quad U^h(1,\cdot)=0.
  	\end{align}
  	Let $X$ be the solution to \eqref{eq.sde}.
  	For each $\bar p\in[1,\infty)$, we put
  	\begin{align*}
  		&\varpi_n(\bar p)=\Big\|\sup_{t\in[0,1]}\Big|\int_0^t(1+\nabla U) [b-b^n](r,X_r)dr\Big|\Big\|_{L_{\bar p}(\Omega)}.
  	\end{align*}
  \end{definition}
  Note that we have changed the definition of $\varpi_n$ from \cite{le2021taming} by replacing $b^{n,h}$ with $b^h$ in \eqref{eqn.uKol}.

  The main result of \cite{le2021taming} (Theorems 2.2 and 2.3 therein) asserts that 
  for any $\bar p\in(1,p)\cap(1,\frac 2d(p\wedge p_0))$ and any $\gamma\in(0,1)$,  there exists a finite constant $N$
   such that
  \begin{align}\label{maine}
  \|\sup_{t\in[0,1]}|X_t^n-X_t|\|_{L_{\gamma\bar p}(\Omega)} \le N \left[\|x^n_0-x_0\|_{L_{\bar p}(\Omega)}+(1/n)^{\frac \alpha2}+ (1/n)^{\frac 12}\log(n) +\varpi_n(\bar p)\right].
  \end{align}
  To obtain estimate, \cite{le2021taming} first utilizes stability results for \eqref{eq.sde} to show that the strong convergence rate is bounded by $\varpi_n(\bar p)$ and the quadrature error of the type
  \begin{align*}
  	\|\sup_{t\in[0,1]}|\int_0^tg(r,X^n_r)(f(r,X^n_r)-f(r,X^n_{k_n(r)}))dr|\|_{L_{\bar p}(\Omega)}
  \end{align*}
  where $f\in \LL^q_p\cap\LL^q_\infty$ and $g\in \LL^q_{1,p}\cap\LL^\infty_\infty$. \cite{le2021taming} then applies stochastic sewing techniques (introduced in \cite{le2018stochastic}) to  obtain the rate $(1/n)^{1/2}\log(n)$ for the quadrature error and to bound $\varpi_n$ by a suitable distance between $b$ and $b^n$.
  
  While \eqref{maine} produces the best available rate in the literature, it comes with an  unnatural constraint on $\bar p$. 
  This restriction is purely technical and is necessary for both stability analysis and stochastic sewing arguments described previously.
  In the more recent article \cite{galeati2022stability}, a stability estimate which is valid for all moments has been obtained.
  In the present article, we utilize the John--Nirenberg inequality (\cref{thm.JohnNirenberg})  to overcome the moment restriction in the stochastic sewing arguments used to estimate $\varpi_n$ and the quadrature error. 
  Our main contribution are the following two results, which remove the moment restrictions from 
  Theorems 2.2 and 2.3 
  of \cite{le2021taming}.
\begin{theorem}\label{thm.main}
	Assume that \crefrange{con.A}{con.B} hold.
	Let $(X_t^n)_{t\in[0,1]}$ be the solution to \eqref{eqn.EMscheme} and $(X_t)_{t\in[0,1]}$ be the solution to \eqref{eq.sde}. Then for any $\bar p\in(1,\infty)$ and any $\gamma>1$,  there exists a finite constant $N(K_1,K_2,K_3,K_4,\alpha,p_0,q_0,p,q,d,\bar p,\gamma)$ such that
	\begin{align}\label{mainenew}
	\|\sup_{t\in[0,1]}|X_t^n-X_t|\|_{L_{\bar p}(\Omega)} \le N \left[\|x^n_0-x_0\|_{L_{\bar p}(\Omega)}+(1/n)^{\frac \alpha2}+ (1/n)^{\frac 12}\log(n) +\varpi_n(\gamma\bar p)\right].
	\end{align}
\end{theorem}
\begin{theorem}\label{thm.alpha}
	Assume that \crefrange{con.A}{con.B} hold with $q_0=\infty$ and $\frac1p+\frac1{p_0}<1$.
	Let $\nu\in[0,1)$ be such that
	\begin{align}\label{con.nu}
		\nu<\frac 32- \frac d{2p}-\frac2q.
	\end{align}
	Then for every $\bar p\in(1,\infty)$, there exists a constant $N$ depending  on $K_1$, $K_2$, $K_3$, $K_4$, $\alpha$, $p_0$, $p$, $q$, $d$, $\bar p$, $\nu$ such that
	\begin{gather}
		\varpi_n(\bar p)\le N\|b-b^n\|_{\LL^q_{-\nu,p}([0,1])}.
		\label{est.anu}
	\end{gather}
\end{theorem}
\cite{le2021taming} also considers the case $\nu=1$ in \cref{thm.alpha}. Our argument also works in this case without much effort. We therefore leave it for interested readers.
  \begin{proposition}\label{prop.Xnob}
  	Let $p\in(1,\infty)$, $q\in(2,\infty)$ and assume that \cref{con.A} holds with $q_0=\infty$ and $\frac1p+\frac1{p_0}<1$. 
  	Let $\bar X$ be a solution to \eqref{eq.sde}.
  	Let $g$ be a function in $\LL^{q}_{p}([0,1])$ and let $\nu\in[0,1)$ such that $\frac{d}{p}+\frac{2}{q}+\nu<2$. Then for any $\bar p\in[1,\infty)$, there exists a constant $N=N(\nu,d,p,q,\bar p)$ such that
  	\begin{align}\label{est.supg}
  		\|\sup_{t\in[0,1]}|\int_0^t g(r, X_r)dr|\|_{L_{\bar p}(\Omega)}\le
  		N \|g\|_{\LL^{q}_{-\nu,p}([0,1])}.
  	\end{align}	
  \end{proposition}
  \begin{proof}
  	By Girsanov transformation, we can assume without loss of generality that $b=0$ (see the argument in the proof of Theorem 5.1 in \cite{le2021taming}).
  	Put $V_t=\int_0^t g(r,X_r)dr$. We note that by Krylov estimate, $\|V_t-V_s\|_{L_m(\Omega)}\les \|g\|_{\LL^q_p} (t-s)^{1-\frac d{2p}-\frac1q}$ for all $m\ge1$ (see inequality (6.11) and Lemma 3.4 from \cite{le2021taming}). Consequently, $V$ is a.s. continuous. The proof of Proposition 6.6 from \cite{le2021taming} shows that
  	\begin{align*}
  		\sup_{s\le t\le 1}(\|\E_s| V_t-V_s|^p\|_{L_\infty(\Omega)})^{1/p}\les \|g\|_{\LL^q_{-\nu,p}([0,1])}.
  	\end{align*}
  	This shows that $V$ is BMO and hence estimate \eqref{est.supg} follows from \cref{thm.JohnNirenberg}.
  \end{proof}
  \begin{proposition}\label{prop.gf}
		Assume that Conditions \ref{con.A}\ref{con.Aholder} and \ref{con.B} hold.
		Let $X^n$ be the solution to \eqref{eqn.EMscheme} and let $f,g$ be measurable functions on $[0,1]\times\Rd$. Assume that $\|f\|_{\LL^q_p([0,1])}=\|g\|_{\LL^\infty_\infty([0,1])}+ \|g\|_{\LL^q_{1,p}([0,1])}= 1$ and $\beta_n(f)=\sup_{0\le j\le n-1}\|f\|_{\LL^q_\infty([j/n,(j+1)/n])}$ is finite.
		Then for any $\bar p\ge1$, there exists a constant $N=N(d,p,q,\bar p)$ such that
		\begin{multline}\label{gfX}
			\left\|\sup_{t\in[0,1]}|\int_0^tg(r,X^n_r) [f(r,X^n_r)-f(r,X^n_{k_n(r)})]dr|\right\|_{L_{\bar p}(\Omega)}
			\\\le N\left[(1/n)^{1-\frac1q}\beta_n(f)+ (1/n)^{\frac \alpha2}+(1/n)^{\frac12}\log(n)\right].
		\end{multline}
	\end{proposition}
	\begin{proof}
		By Girsanov transformation, we can assume without loss of generality that $b=0$ (see the argument in the proof of Theorem 5.1 in \cite{le2021taming}). 
		Define
		\begin{align*}
			V_t=\int_0^tg(r,X^n_r) [f(r,X^n_r)-f(r,X^n_{k_n(r)})]dr.
		\end{align*}
		Proposition 5.12 of \cite{le2021taming} and its proof shows that for every $v+4/n\le s\le t\le 1$,
		\begin{align*}
			\|\E_v|V_t-V_s|^p\|_\infty^{1/p}\les [(1/n)^{\frac \alpha2}+(1/n)^{\frac12}\log(n)].
		\end{align*}
		By assumption and H\"older inequality, we have for every $s\le t$
		\begin{align*}
			|V_t-V_s|\les\|f\|_{\LL^q_\infty([s,t])}  (t-s)^{1-\frac1q} .
		\end{align*}
		Combining the previous two estimates yields that for every $s\le t$
		\begin{align*}
			\|\E_s|V_t-V_s|^p\|_\infty^{1/p}\les [(1/n)^{\frac \alpha2}+(1/n)^{\frac12}\log(n)+(1/n)^{1-\frac1q}\beta_n(f)].
		\end{align*}
		Since $V$ is continuous, from \cref{prop.s2S}, it is BMO (in fact VMO) and hence by applying \cref{thm.JohnNirenberg}, we obtain \eqref{gfX}.		
	\end{proof}

\begin{proof}[\bf Proof of \cref{thm.main}]
	The stability argument in \cite{le2021taming} comes with a restriction on the moment and we must replace it by the recent stability estimate from \cite{galeati2022stability}. Indeed, from Section 3.4.1 of the aforementioned reference, we have
	\begin{align}\label{est.stability}
		\|\sup_{t\in[0,1]}|X_t-X^n_t|\|_{L_{\gamma\bar p}(\Omega)}\les\|x_0-x^n_0\|_{L_{\bar p}(\Omega)}+\|\sup_{t\in[0,1]}|V_t|\|_{L_{\bar p}(\Omega)},
	\end{align}
	where
	\begin{align*}
		V_t&=\int_0^t\left(\frac12[(R^2+\sigma)(R^2+\sigma)^*- \sigma \sigma^*]:D^2U+R^1\cdot(I+\nabla U)\right)(r,X^n_r)dr
		\\&\quad+\int_0^t[R^2(I+\nabla u)](r,Y_r)dB_r,
	\end{align*}
	\begin{align*}
		R^1_t=b^n(t,X^n_{k_n(t)})-b(t,X^n_t), \quad R^2_t=\sigma(t,X^n_{k_n(t)})-\sigma(t,X^n_t).
	\end{align*}
	Since $\sigma$ is H\"older continuous, the moments of terms with $R^2$ are bounded by a constant multiple of $(1/n)^{\alpha/2}$ (see Section 7 of \cite{le2021taming} for some analogous estimates). To treat the term with $R^1$, we note that
	\begin{multline*}
		|\int_0^tR^1\cdot(I+\nabla U)(r,X^n_r)dr|
		\le 
		|\int_0^t(b^n(r,X^n_{k_n(r)})-b^n(r,X^n_r))\cdot(I+\nabla U)(r,X^n_r)dr|
		\\+|\int_0^t(b^n-b)(r,X^n_r)\cdot(I+\nabla U)(r,X^n_r)dr|. 
	\end{multline*}
	To treat the first term, we apply \cref{prop.gf}, regularity of $U$ (Lemma 7.1 of \cite{le2021taming}) and \cref{con.B} to have
	\begin{multline*}
		\left\|\sup_{t\in[0,1]}|\int_0^t(b^n(r,X^n_{k_n(r)})-b^n(X^n_r))\cdot(I+\nabla U)(r,X^n_r)dr|\right\|_{L_{\bar p}(\Omega)}
		\\\les (1/n)^{1-\frac1q}\beta_n(b^n)+ (1/n)^{\frac \alpha2}+(1/n)^{\frac12}\log(n)\les  (1/n)^{\frac \alpha2}+(1/n)^{\frac12}\log(n).
	\end{multline*}
	Moment of the second term is directly related to $\varpi_n(\bar p)$ through \cref{def.rate}.
	This leads us  to the following estimate
	\begin{align*}
		\|\sup_{t\in[0,1]}|V_t|\|_{L_{\bar p}(\Omega)}\les  (1/n)^{\frac \alpha2}+(1/n)^{\frac12}\log(n)+\varpi_n(\bar p).
	\end{align*}
	Combining with \eqref{est.stability}, we obtain \eqref{mainenew}.
\end{proof}

\begin{proof}[\bf Proof of \cref{thm.alpha}]
	The proof follows in  exactly  the same way as the proof of Theorem 2.3 in \cite{le2021taming} (Section 7 therein). The restriction $\bar p<p$ there is now lifted thanks to \cref{prop.Xnob}.
\end{proof}


\appendix
\section{Auxiliary results} 
\label{sec:auxiliary_results}
	\begin{lemma}[Garsia's upcrossing lemma, \cite{MR0448538,MR341601}]\label{lem.upcrossing}
		Let $(X_t)_{t\in[0,\tau]}$ be a right continuous adapted process with left limits and let $s$ be a fixed time in $[0,\tau]$. Suppose that there is a non-negative integrable random variable $U$ such that
		\begin{align}\label{con.garsia}
			\E_S|X_T-X_{S-}|\le \E_SU
		\end{align}
		for any pair $S,T$ of stopping times with $s\le S\le T\le \tau$. 
		Let $Y$ be an $\cff_s$-random variable and define $X^*=\sup_{s\le r\le \tau}|X_r-Y|$. Then for every $\alpha,\beta>0$, one has
		\begin{align}
			\beta\PP_s(X^*\ge \alpha+\beta)\le\E_s(U\1_{(X^*\ge \alpha)}).
		\end{align}
	\end{lemma}
	\begin{proof}
		We adopt the arguments from  \cite{MR1299529}.
		Let $\alpha,\beta>0$ be given and $G\in \cff_s$. 
		We set  $X_t=X_\tau$ for $t>\tau$ and define  \khoa{here, we used right-continuity of the filtration. Can be removed if the process is continuous}
		\begin{align*}
			S= \inf\{t\ge s:|X_t-Y|\ge \alpha\}, \quad T=\inf\{t\ge s:|X_t-Y|\ge \alpha+\beta\},
		\end{align*}
		with the standard convention that $\inf(\emptyset)=\infty$.
		Clearly $S$ and $T$ are stopping times and $s\le S\le T$.
		We also have from the above definitions, 
		\begin{align}\label{inclusion}
			\{X^*\ge\alpha+\beta\}\subset\{|X_T-X_{S-}|\ge \beta,\,|X_S-Y|\ge \alpha\}.
		\end{align}
		It follows that
		\begin{align*}
			\E(\1_{(X^*\ge \alpha+\beta)}\1_G)&\le \E(\1_{|X_T-X_{S-}|\ge \beta}\1_{(|X_S-Y|\ge \alpha)}\1_G)
			\le \frac 1 \beta\E(|X_T-X_{S-}|\1_{(|X_S-Y|\ge \alpha)}\1_G)
			\\&\le \frac 1 \beta\E(U\1_{(|X_S-Y|\ge \alpha)}\1_G)
		\end{align*}
		which implies the result.
	\end{proof}
	We note that right-continuity of the filtration is necessary so that $S,T$ defined in the previous proof are stopping times. 
	In addition, the inclusion \eqref{inclusion} does not hold if one replaces $X_{S-}$ by $X_S$ in \eqref{con.garsia}. These technical conditions become irrelevant when dealing with continuous processes.
	\begin{lemma}[Energy inequality]\label{lem.energy}
		Let $c$ be a deterministic constant and $(A_t)_{t\ge0}$ be an adapted, right-continuous, non-decreasing process. Let $\tau>0$ be fixed and suppose that
		\begin{align}\label{con.ener}
			\|\E_{S}(A_\tau-A_{S-})\|_\infty \le c \text{ for every stopping time }S\le \tau.
		\end{align}
		Then for every $s\in[0,\tau]$ and every integer $p\ge1$,
		\begin{align}\label{est.energy.p}
			\|\E_s (A_\tau-A_s)^p\|_\infty	\le p! c^p. 
		\end{align}	
	\end{lemma}
	\begin{proof}
		When $A_t$ takes the specific form $\int_0^t \beta(r)dr$ for some $\beta\ge0$, this result deduces to the Khasminskii's lemma (\cite{MR123373}). In the general form, it is known as energy inequality and can be found in \cite{MR0205288,MR1232016}. Our statement is for processes over finite time intervals which  differs from previous ones and needs  justifications.

		Let $s\in\left[0,\tau\right]$ be fixed and $G$ be an event in $\cff_s$. For each $r\ge0$,   define $\tilde A_{r}=\1_G(A_{(r+s)\wedge \tau}-A_s)$. The process $\tilde A$ is adapted with respect to the filtration $\tilde\cff:=\{\cff_{r+s}\}_{r\ge0}$, right-continuous, satisfies $\tilde A_0=0$ and $\|\E_{S}(\tilde A_\tau-\tilde A_{S-})\|_\infty \le c$ for all $\tilde \cff$-stopping times $S$. Applying Theorem 4 of \cite{MR1232016} to the process $\tilde A$, we get that $\E(\1_G(A_\tau-A_s)^p)\le p!c^p$.
		Since $G$ is arbitrary, this implies \eqref{est.energy.p}.
	\end{proof}
	
\section*{Acknowledgment} 
	The author thanks Peter Friz for motivating discussions on exponential integrability and his interest in this work. He is grateful to Stefan Geiss for explaining the works \cite{MR4092418,geiss2020riemann,MR2136865} and equivalent definitions of BMO processes.
\section*{Funding} 
	The author was supported by Alexander von Humboldt Research Fellowship during the early stage of this work.
\bibliographystyle{alpha}
\bibliography{biblio}
\end{document}